\documentclass{article}
\usepackage{amsmath} 
\usepackage{amsthm} 
\usepackage{amsfonts} 
\usepackage{amssymb} 
\usepackage{tikz-cd} 
\usepackage[left=1.5in,right=1.5in,top=1in,bottom=1in]{geometry}
\usepackage{hyperref} 
\usepackage{cleveref} 
\usepackage{mathrsfs} 
\usepackage{mathtools} 

\newtheorem{THM}{Theorem}
\newtheorem{COR}[THM]{Corollary}

\newtheorem{thm}{Theorem}
\newtheorem{lem}[thm]{Lemma}
\theoremstyle{definition}
\newtheorem{quest}{Question}

\DeclareMathOperator{\aut}{Aut}
\DeclareMathOperator{\out}{Out}
\newcommand{\cato}{\operatorname{CAT}(0)}

\usepackage{import}

\begin{document}
\title{\Large Some new CAT(0) free-by-cyclic groups}
\author{Rylee Alanza Lyman}
\date\today
\maketitle

\begin{abstract}
We show the existence of
several new infinite families of polynomially-growing automorphisms of free groups
whose mapping tori are CAT(0) free-by-cyclic groups.
Such mapping tori are \emph{thick,} and thus not relatively hyperbolic.
These are the first families comprising infinitely many examples 
for each rank of the nonabelian free group;
they contrast strongly with Gersten's example of a thick free-by-cyclic group
which cannot be a subgroup of a CAT(0) group.
\end{abstract}

Let $F_3 = \langle a,b,c \rangle$ denote, with basis,
a free group of rank $3$.
Consider the following automorphisms of $F_3$
\[
	\Theta \begin{dcases}
		a \mapsto a \\
		b \mapsto ba \\
		c \mapsto ca^2
	\end{dcases}
	\qquad
	\Phi \begin{dcases}
		a \mapsto a \\
		b \mapsto a^{-1}ba \\
		c \mapsto a^{-2}ca^2
	\end{dcases}
	\qquad
	\Psi \begin{dcases}
		a \mapsto a \\
		b \mapsto aba \\
		c \mapsto a^2ca^2.
	\end{dcases}
\]
In \cite{GerstenFreeByZ}, Gersten showed that the \emph{mapping torus} of $\Theta$, 
in other words the \emph{free-by-cyclic} group $F_3\rtimes_\Theta\mathbb{Z}$,
cannot act properly by semi-simple isometries on a $\cato$ metric space---in particular,
$F_3\rtimes_\Theta\mathbb{Z}$ cannot be a subgroup of a $\cato$ group.

In contrast, the purpose of this note is to show that the mapping tori of $\Phi$, $\Psi$,
and indeed all other examples of their kind are themselves $\cato$ groups.

\begin{THM}
	\label{symmetric}
	Let $\Phi\colon F_n \to F_n$ be a polynomially-growing, symmetric automorphism.
	There exists an integer $k \ge 1$ such that the mapping torus of $\Phi^k$ 
	acts geometrically on a $\cato$ $2$-complex.
	The power $k$ satisfies $k < n!$.
	If the automorphism $\Phi$ is upper triangular, then the mapping torus of $\Phi$ is a $\cato$ group.
\end{THM}

\begin{COR}
	\label{palindromic}
	Let $\Phi\colon F_n \to F_n$ be a polynomially-growing, palindromic automorphism.
	There exists an integer $k \ge 1$ such that the mapping torus of $\Phi^k$
	acts geometrically on a $\cato$ $2$-complex.
	The power $k$ satisfies $k < n!$.
	If the automorphism $\Phi$ is upper triangular, then the mapping torus of $\Phi$ is a $\cato$ group.
\end{COR}

An automorphism $\Phi\colon F_n \to F_n$ is \emph{polynomially-growing} if for all $g \in F_n$,
the word length of $\Phi^k(g)$ grows at most polynomially in $k$.
Fix a free basis $x_1,\dotsc,x_n$ for $F_n$.
The automorphism $\Phi$ is \emph{symmetric} with respect to this basis if it
permutes the conjugacy classes of the $x_i$.
To wit, in this case there exist words $w_1,\dotsc,w_n$ in the $x_i$ such that
for each $i$ satisfying $1 \le i \le n$, we have
$\Phi(x_i) = w_i^{-1}x_jw_i$ for some $j$ satisfying $1 \le j \le n$.
Given a word $w$ in our free basis, write $\bar{w}$ for the \emph{reverse} of $w$,
e.g.~we have $\overline{x_1x_2} = x_2x_1$.
The automorphism $\Phi$ is \emph{palindromic} with respect to the
basis $x_1,\dotsc,x_n$ if for each $i$ satisfying $1 \le i \le n$, we have
$\Phi(x_i) = \bar w_ix_jw_i$ for some $j$ satisfying $1 \le j \le n$.
In particular, elements of our free basis are sent to \emph{palindromes}---words
spelled the same forwards and backwards.
Finally, in both of the above cases, the automorphism is \emph{upper triangular}
when we always have $i = j$,
and for each $i$ satisfying $1 \le i \le n$,
the word $w_i$ may be spelled using only the basis elements $x_1,\dotsc,x_i$.

\Cref{palindromic} is a corollary of the following theorem.

\begin{THM}
	\label{freeproductoffinite}
	Let $A$ be a finite group, let $W_n = A * \dotsb * A$
	denote the free product of $n$ copies of $A$,
	and let $\Phi\colon W_n \to W_n$ be a polynomially-growing automorphism.
	There exists an integer $k \ge 1$ such that the mapping torus of $\Phi^k$
	acts geometrically on a $\cato$ $2$-complex.
\end{THM}
	
As the similarities and distinctions between the automorphisms $\Theta$, $\Phi$ and $\Psi$ above illustrate,
free-by-cyclic groups form a varied and interesting class of finitely-presented groups
whose properties remain far from completely understood.
For instance, it is not known in general when a free-by-cyclic group
admits a geometric action on a $\cato$ space.
For a long time it was thought that perhaps a free-by-cyclic group 
would provide the first example of a hyperbolic group that
cannot act geometrically on a $\cato$ space.
Recently Hagen and Wise \cite{HagenWise,HagenWise2} showed that in fact
hyperbolic free-by-cyclic groups act geometrically on $\cato$ cube complexes,
and are thus \emph{virtually special.}
The free-by-cyclic groups we consider are not relatively hyperbolic:
Hagen \cite{HagenThick} notes that a result of Macura \cite{Macura}
implies that mapping tori of polynomially-growing automorphisms
are \emph{thick} in the sense of Behrstock--Dru\c{t}u--Mosher \cite{BehrstockDrutuMosher},
and thick groups are not nontrivially relatively hyperbolic.
A famous theorem originally due to Gautero and Lustig \cite{GauteroLustig} and 
independently given new proofs by Ghosh \cite{Ghosh} and Dahmani--Li \cite{DahmaniLi}
says that free-by-cyclic groups are hyperbolic relative to a canonical
collection of thick free-by-cyclic subgroups.
These are the \emph{subgroups of polynomial growth} defined in \cite{Levitt}.

The question of which free-by-cyclic groups are $\cato$ remains 
an interesting open problem in general \cite{BridsonProblems}.
This paper represents a major contribution to this question when
the rank of the free group is allowed to vary
and the free-by-cyclic group is assumed to be thick. 
Every $F_2$-by-$\mathbb{Z}$ group can be represented as a non-positively curved
punctured-torus bundle over the circle,
so every $F_2$-by-$\mathbb{Z}$ group is $\cato$.
In fact, Button and Kropholler \cite{ButtonKropholler} proved that every $F_2$-by-cyclic group
is the fundamental group of a non-positively curved cube complex of dimension $2$.

Brady and Soroko ask whether a free-by-cyclic group is $\cato$
if and only if it is virtually special \cite{BradySoroko}.
Our $\cato$ spaces, while 2-dimensional, are in general not cube complexes,
so a reader interested in Brady--Soroko's question may wish to investigate the following question.
\begin{quest}
	May these $\cato$ free-by-cyclic groups be cocompactly cubulated?
	Is the resulting cube complex virtually (co-)special?
\end{quest}

In another direction, \Cref{freeproductoffinite} suggests a more general statement might be true.
\begin{quest}
	When $W$ is a virtually free group with finite abelianization,
	are thick $W$-by-cyclic groups $\cato$?
\end{quest}

Our proof of \Cref{freeproductoffinite} is somewhat tailored to the case 
of free products of copies of a single finite group,
but perhaps there is some way to remove this restriction.

The $\cato$ 2-complex we construct is somewhat reminiscent of a graph manifold in construction.
One begins at the first level with a torus and progressively attaches cylinders
in such a way that the complex remains non-positively curved.
This is the space considered in a special case by Samuelson in \cite{Samuelson},
and the gluing is informed by Levitt's \emph{cyclic hierarchy} for thick free-by-cyclic groups.
The additional assumptions needed for our theorems assure
that the glued object is non-positively curved.

Here is the organization of this note. 
Gersten's non-example $\Theta$ is too cute to pass up;
in \Cref{Examples}
we sketch his proof very briefly,
explain the construction of Bridson--Haefliger 
and work an example of \Cref{palindromic} to illustrate the proof of \Cref{freeproductoffinite}.
A reader in a great hurry could skip this section
and proceed directly to the proof in \Cref{ProofOfMainTheorems}

\paragraph{Acknowledgments}
The author wishes to thank her advisor, Kim Ruane, for her unfailing enthusiasm in this project
and the author's development as a mathematician;
Santana Afton, Mladen Bestvina, Mark Hagen, and Robert Kropholler for many helpful discussions
about preliminary versions of these arguments; and Matthew Zaremsky for asking about symmetric automorphisms.

\section{(Non-)Examples and the Construction}
\label{Examples}
As a warm-up,
it will be instructive to give Gersten's non-example and compare it
with an example of \Cref{palindromic}.
In this section we also give an exposition of the construction of the CAT(0) 2-complex.

\subsection{Gersten's non-Example}
Gersten's example concerns the following automorphism $\Theta\colon F_3\to F_3$.
We write $F_3 = \langle a,b,c \rangle$.
\[ \Theta \begin{dcases} a \mapsto a \\
		b \mapsto ba \\
		c \mapsto ca^2
	\end{dcases}\quad F_3\rtimes_\Theta\mathbb{Z} = \langle a,b,c,t \mid
	tat^{-1} = a,\ tbt^{-1} = ba,\ tct^{-1} = ca^2 \rangle
\]
Gersten's first observation is to rewrite this group as a double
HNN extension of $\langle a,t\rangle \cong \mathbb{Z}^2$ with $b$ and $c$
as stable letters. He does this by rewriting the relators.
\[
	tbt^{-1} = ba \leadsto b^{-1}tb = at\quad\text{ and }\quad
	tct^{-1} = ca^2 \leadsto c^{-1}tc = a^2t \]
This observation generalizes:
every thick free-by-cyclic group has a finite-index subgroup that admits a
\emph{cyclic hierarchy,} a repeated graph-of-groups decomposition with
cyclic edge stabilizers and thick free-by-cyclic groups
of lower rank as vertex stabilizers, terminating with $\mathbb{Z}^2$.
Levitt records this fact as~\cite[Definition 1.1]{Levitt}.
This allows for arguments by induction.

We return to Gersten's proof.  
Suppose, aiming for a contradiction, that $F_3\rtimes_\Theta\mathbb{Z}$
acts properly by semi-simple isometries on a $\cato$ metric space $X$.
By the Flat Torus Theorem~\cite[Theorem II.7.1, p. 244]{theBible},
there is an isometrically embedded Euclidean plane $Y\subset X$.
This plane is preserved by $H = \langle a,t \rangle$, which acts on
$Y$ by translation, and the quotient $Y/H$ is a $2$-torus.

Fix a point $p \in Y$.
The content of the HNN extension is that in $F_3\rtimes_\Theta\mathbb{Z}$,
$t$, $at$ and $a^2t$ are all conjugate, so in the action of
$H$ on $Y$, these elements have the same translation length.
Thus there is a circle of radius $d(p,t.p)$ in $Y$ centered at $p$
that meets the points $t.p$, $at.p$ and $a^2t.p$.
But on the other hand, these three points lie on an axis
for the translation action of $a$. But in Euclidean geometry, a straight
line cannot meet a circle in three points. See \Cref{gerstenfig}.

\begin{figure}[!ht]
	\centering
	    \def\svgwidth{\columnwidth}
	        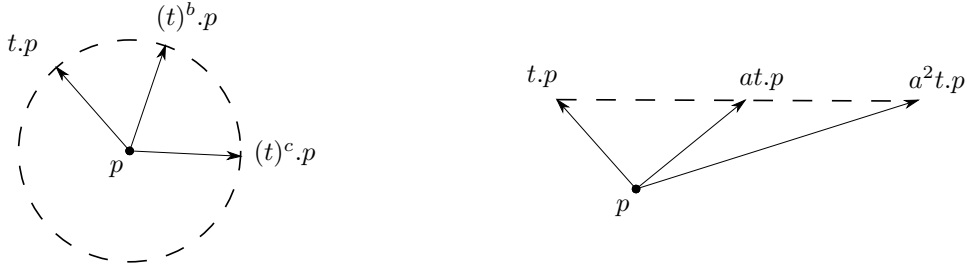
	
	\caption{Gersten's example would force a line to intersect a circle in three points.}
	\label{gerstenfig}
\end{figure}

The lesson here is that in order for an HNN extension of a
$\cato$ group with cyclic associated subgroups to be a $\cato$
group, there must be a ``good reason'' for the generators 
of the associated cyclic subgroups
to have the same translation length.

\subsection{Bridson--Haefliger's Construction}
Bridson and Haefliger give a general construction providing
a sufficient condition for an HNN extension of a $\cato$
group to be a $\cato$ group. Because the geometry of this
space will be important to our arguments, we describe
their construction in the setting where the
associated subgroups are infinite cyclic.

\begin{thm}[\cite{theBible} Proposition II.11.21, p. 358]
	\label{construction}
	Let $H$ be a group acting properly and cocompactly
	on a $\cato$ space $X$. Given $x$ and $y$ infinite
	order elements of $H$ whose translation lengths
	on $X$ are equal, there is a $\cato$ space $Y$
	on which the HNN extension $G = H\ast_{x^t = y}$
	acts properly and cocompactly.
\end{thm}

Informally, the construction proceeds by ``blowing up'' the
Bass--Serre tree $T$ for the HNN extension. For each vertex
of $T$, $Y$ contains an isometric copy of $X$. When two vertices
share an edge, there is a \emph{strip,} that is, a space $S := \mathbb{R}\times[0,1]$ glued
in, with $\mathbb{R}\times\{0\}$ glued to one copy of $X$
along an axis for $x$, and $\mathbb{R}\times\{1\}$ glued
to another copy of $X$ along an axis for $y$. See \Cref{stripfig}.

In the case where $X$ is the universal cover of a space with
fundamental group $H$, one might imagine attaching a cylinder
to $X/H$ with one end attached along a loop representing $x$ and the
other along a loop representing $y$ in $\pi_1(X/H)$. If this is
done carefully, the universal cover of the resulting space is $Y$.

Formally, fix geodesic axes $\gamma$ and $\eta$ for the actions
of $x$ and $y$ on $X$, respectively. We think of
$\gamma$ and $\eta$ as isometric embeddings of $\mathbb{R}$
into $X$. Let $\alpha$ be the translation length of both 
$x$ and $y$ in $X$.
Recall that the vertices of the Bass--Serre tree $T$
correspond to cosets of $H$ in $G$ and edges of $T$
correspond to cosets of $K = \langle x \rangle$.
The vertices $gH$ and $gtH$ are connected by the edge
$gK$ in $T$.
Let $K$ act on $S$ by translation by $\alpha$ in the first factor.
The $\cato$ space $Y$ is a quotient of
the disjoint union $G \times X \cup G\times S$ by
the equivalence relation generated by the following,
where $g \in G$, $h \in H$, $x \in K$, $p \in X$,
$t \in \mathbb{R}$ and $\theta \in [0,1]$.
\begin{enumerate}
	\item $(gh,p) \sim (g,h.p)$
	\item $(gx,t,\theta) \sim (g,x.t,\theta)$
	\item $(g,\gamma(t)) \sim (g,t,0)$
	\item $(gt,\eta(t)) \sim (g,t,1)$
\end{enumerate}
The group $G$ acts on $Y$ by multiplication in the labels,
and it is easy to see that $Y$ contains distinct, isometrically
embedded copies of $X$ for each coset $G/H$, and likewise
for copies of $S$ indexed by the cosets $G/K$.

\begin{figure}[!ht]
	\centering
	    \def\svgwidth{\columnwidth}
	        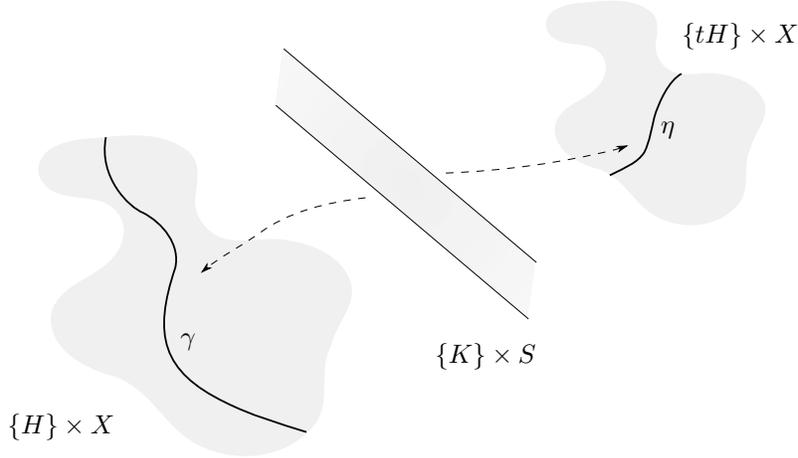
	
	\caption{Gluing two copies of $X$ via a strip}
	\label{stripfig}
\end{figure}

\subsection{A Palindromic Automorphism of \texorpdfstring{$F_3$}{F3}}
\label{F3Example}
Let $W_n$ be the free product of $n$ copies of a cyclic group of order $2$,
\[ W_n = \underbrace{C_2*\dotsb*C_2}_{n\text{ copies}} = \langle a_1,\dotsc, a_n \mid a_i^2 =1,\ 1\le i \le n\rangle. \]
We have a homomorphism $W_n \to C_2$ sending each $a_i$ to the generator of $C_2$.
This map splits by sending $C_2$ to $a_1$,
and the kernel is free of rank $n-1$,
so $W_n = F_{n-1}\rtimes C_2$.
A free basis for the kernel is $a_1a_2,\dotsc,a_1a_n$.
In our example,
$n = 4$; we will write
$F_3 = \langle x,y,z \rangle$, and write $a$ for the generator
of the $C_2$ factor. We have $a^{-1}xa = x^{-1}$, and similarly 
for $y$ and $z$. Automorphisms of $F_{n-1}$ that commute
with the conjugation action of $a$ send basis elements to
palindromes.
Consider the following palindromic automorphism of $F_3$.
\[
	\Phi\begin{dcases}
		x \mapsto x \\
		y \mapsto xyx \\
		z \mapsto yzy
	\end{dcases}
	\qquad
	F_3\rtimes_\Phi\mathbb{Z} = 
	\langle x,y,z,t \mid [x,t]=1,\,(x^{-1}t)^y = xt,\,(y^{-1}t)^z = yt\rangle
\]
We use the exponential notation for conjugation $x^y = y^{-1}xy$.
Let $b = ax$, $c = ay$, and $d = az$ be the generators for $W_4$ as a free product of finite groups.
Setting $\Phi(a) = a$, we get an automorphism, still
called $\Phi$ 
\[
	\Phi\begin{dcases}
		a\mapsto a \\
		b \mapsto b \\
		c \mapsto bacab \\
		d \mapsto cadac.
	\end{dcases}
\]

Our aim is to inductively apply \Cref{construction}
to show that $F_3\rtimes_\Phi\mathbb{Z}$ is the
fundamental group of a $\cato$ $2$-complex. Along the way
we will also show the complex admits a compatible action of
$a$, so the resulting mapping torus, $W_4\rtimes_\Phi\mathbb{Z}$
acts properly and cocompactly on the same space.
Above we have rewritten the presentation for $F_3\rtimes_\Phi\mathbb{Z}$
to make the hierarchy clearer. Write $G_0 = \langle x,t\rangle \cong
\mathbb{Z}^2$, $G_1 = G_0\ast_{(xt)^y = x^{-1}t}$ and $G_2 =
G_1\ast_{(yt)^z = y^{-1}t} = F_3 \rtimes_\Phi \mathbb{Z}$.
Write $K_0 = \langle xt \rangle$ and $K_1 = \langle yt\rangle$,
respectively.

\paragraph{Step One.} The first $\cato$ space, $X_0$,
for $G_0$ to act on is the Euclidean plane by translation.
Letting $\vec x$ and $\vec t$ be the translation vectors 
for $x$ and $t$, notice that the translation lengths
of $xt$ and $x^{-1}t$ are equal to the lengths of the diagonals
of the parallelogram determined by $\vec x$ and $\vec t$.
This implies that $xt$ and $x^{-1}t$ have the same translation
length exactly when $\vec x$ and $\vec t$ are orthogonal.

Choosing $X_0$ so that $\vec x$ and $\vec t$ are orthogonal,
$X_0$ admits an isometric action of $a$ by reflection
across a fixed geodesic axis for $t$. Choose an axis
$\gamma$ for $xt$ and $\eta := a.\gamma$ for $x^{-1}t$. 
With this data, we apply \Cref{construction}
to yield a new $\cato$ space $X_1$ on which $G_1$ acts 
properly and cocompactly. 

\paragraph{Step Two.}
We extend $a$ to an isometry of $X_1$:
$a$ acts on the copy of $X_0$ corresponding to the identity
coset of $G_1/G_0$ as in the previous paragraph. If $h \in G_0$, $a$ takes
$h.\gamma$ to $h^a.\eta$, and vice versa, so we extend our definition of $a$ so that it
swaps the associated strips $S = \mathbb{R}\times[0,1]$
and sends $(s,\theta)$ to $(s,1-\theta)$.
More generally, $a$ takes $gG_0\times X_0$ to
$g^aG_0\times X_0$, takes $gK_0\times S$ to $g^ay^{-1}K_0$
and acts as above in the $X_0$ and $S$ factors.
One checks that because $a$ is an isometry of the pieces
and respects the gluing, this defines an isometry of $X_1$.

\paragraph{Step Three.}
This done, notice that $(yt)^a = y^{-1}t$, so these elements
must have the same translation length in $X_1$! Now we repeat:
applying \Cref{construction} one more time yields a $\cato$
space on which $G_2$ acts properly and cocompactly. In fact,
an identical argument as above allows us to again extend $a$
to an isometry of $X_2$, as desired.

Thus we see that our example satisfies the conclusions of 
\Cref{palindromic} and \Cref{freeproductoffinite}.

\section{Proof of the Main Theorem}
\label{ProofOfMainTheorems}

Let $A$ be a finite group or $\mathbb{Z}$, and write $W_n$ for the free product of $n$ copies of $A$.
We are interested in polynomially-growing automorphisms $\Phi\colon W_n \to W_n$
which permute the conjugacy classes of the $A$ in $W_n$.
(This is automatic if $A$ is finite, and is the assumption that $\Phi$ is symmetric if $A = \mathbb{Z}$.)
The key technical lemma we need is the following.
\begin{lem}
	\label{keylemma}
	Given a polynomially-growing automorphism $\Phi\colon W_n \to W_n$ as above, 
	there is an automorphism $\Phi'\colon W_n \to W_n$ in the same outer class as $\Phi$
	which is a root of an automorphism which is upper triangular
	with respect to some free product decomposition of $W_n$.
\end{lem}

Let us recall that if $\Phi$ and $\Phi'$ define the same outer automorphism of $W_n$,
then $W_n\rtimes_\Phi\mathbb{Z}$ and $W_n\rtimes_{\Phi'}\mathbb{Z}$ are isomorphic,
so there is no loss in passing from one to the other.
\emph{Upper triangular} here means that there is a free product decomposition 
$W_n = B_1*\dotsb*B_n$,
where each $B_i$ is conjugate to one of the original $A$,
and there exist $w_i \in B_1 * \dotsb * B_{i-1}$ such that
$\Phi(b_i) = w_i^{-1}b_iw_i$ for $b_i \in B_i$.

Assuming the lemma for now, we prove the main theorems.

\begin{proof}[Proof of \Cref{symmetric}]
	Suppose $\Phi\colon F_n \to F_n$ is a polynomially-growing symmetric automorphism.
	Then after replacing $\Phi$, 
	\Cref{keylemma} yields a basis $x_1,\dotsc,x_n$ for $F_n$
	and an automorphism $\Phi\colon F_n \to F_n$ which is upper-triangular.
	We have
	\[ F_n\rtimes_\Phi\mathbb{Z} = \langle x_1,\dotsc,x_n, t 
	\mid tx_kt^{-1} = w_k^{-1}x_kw_k \rangle, \]
	where each $w_k \in \langle x_1,\dotsc x_{k-1}\rangle$,
	and $w_1 = 1$.
	Note that the relation $tx_kt^{-1} = w_k^{-1}x_kw_k$ can be rewritten as
	$x_k^{-1}w_ktx_k = w_kt$, yielding a hierarchy for $F_n\rtimes_\Phi\mathbb{Z}$
	as an iterated HNN extension.
	The first stage is the base group $\langle x_1,t\rangle \cong \mathbb{Z}^2$.
	At each stage, the hypothesis of \Cref{construction} are obviously satisfied,
	so iteratively applying \Cref{construction} beginning with any geometric action of
	$\langle x_1,t\rangle$ on the Euclidean plane proves the result.
\end{proof}

To prove \Cref{freeproductoffinite}, we need a bit of group theory.
Let $A$ be a finite group and write $W_n$ for the free product of $n$ copies of $A$.
If $a \in A$, write $a_i$ for $a$ in the $i$th free factor.
There is a surjective homomorphism $W_n \to A$ sending each $a_i$ to $a$.
This map splits: send $a \in A$ to $a_1$.
The kernel is free of rank $(|A| - 1)(n-1)$, a free basis is given by
\[ a_1^{-1}a_2,\dotsc, a_1^{-1}a_n \]
as $a \in A \setminus \{1\}$ varies.
Thus $W_n \cong F \rtimes A$ for a free group $F$.

\begin{proof}[Proof of \Cref{freeproductoffinite}]
	Fix  a finite group $A$. We proceed by
	induction on $n$, the Kurosh rank of $W_n$. The base case is
	$n=2$. Since upper triangular automorphisms of $W_2$ are inner,
	we may consider the action on the (metric) product 
	$T \times \mathbb{R}$, where $T$ is a regular $|A|$-valent tree 
	on which $W_2$ acts geometrically with two orbits of vertices.
	
	So assume that for $k < n$ and for all polynomially-growing
	automorphisms $\Phi\colon W_k \to W_k$, the conclusions of the
	theorem hold. 
	By \Cref{keylemma}, we may
	without loss of generality assume that $\Phi$ is upper triangular.
	We assume (perhaps after replacing $\Phi$ by a power)
	that the mapping torus of
	$\Phi|_{ A_1*\dots*A_{n-1}}$ acts properly and
	cocompactly on a $\cato$ $2$-complex $X$.
	
	Recall our notation from above: for $a \in A$, we write $a_i$
	for the image of $a$ in $A_i$. We have $W_n = F\rtimes A$,
	where $F = \langle a_1^{-1}a_n \mid 2 \le i \le n\text{ and }a
	\in A\setminus\{1\}\rangle$. 
	By \Cref{keylemma},
	there is $w_n \in A_1*\dotsb*A_{n-1}$ 
	such that $\Phi(a_n) = w_n^{-1}a_nw_n$ for all $a_n \in A_n$.
	We may also assume $\Phi(a_1) = a_1$ for all $a_1 \in A_1$.
	We need that $w_n \in F\cap W_{n-1}$. This can be arranged by composing
	$\Phi$ with the inner automorphism corresponding to conjugation
	by some $a \in A_1$. This does not change the isomorphism type
	of the mapping torus of $\Phi|_{A_1*\dotsb*A_{n-1}}$, and we continue
	to work with $X$.

	If $A$ is not abelian, it may no longer be the case that $\Phi(a_1) = a_1$
	for all $a_1 \in A_1$. Restore this property by replacing $\Phi$ by a power.

	This done, notice that $ta_1^{-1}a_nt^{-1} = a_1^{-1}w_n^{-1}a_nw_n = 
	(w_n^{-1})^{a_1}a_1^{-1}a_nw_n$. 
	Recall our exponential notation $x^y = y^{-1}xy$.
	This implies that $a_1^{-1}a_n$,
	thought of as the stable letter for our HNN extension,
	conjugates $(w_nt)^{a_1}$ to $w_nt$.
	Since $a_1$ is an isometry of $X$, as in \Cref{F3Example},
	we may apply \Cref{construction} for each $a_n \in A_n$.
	We do this by first fixing an axis $\gamma$ for the action of $w_nt$ on $X$,
	and then using $a_1.\gamma$ as $a_1 \in A_1$ varies to work as the
	geodesic axes of interest.
	This yields a $\cato$ space $Y$ that
	$F\rtimes_{\Phi|_F}\mathbb{Z}$ acts on geometrically. We check that
	once again, the isometric actions of $a_1 \in A_1$ on $X$
	also extend to isometries of $Y$, proving the claim.
\end{proof}

\begin{proof}[Proof of \Cref{palindromic}]
	As we saw in the example in \Cref{Examples}, 
	palindromic automorphisms of free groups arise as the restriction to a finite-index free subgroup
	of automorphisms $\Phi\colon W_n \to W_n$ in the case where $A$ is cyclic of order two.
\end{proof}

\subsection{Train Track Maps}

To complete the proof, it remains to prove \Cref{keylemma}.
In this subsection we assume knowledge of the relative train track maps of \cite{BestvinaHandel}
as generalized to graphs of groups in \cite{MyThesis}.

Let $(\Gamma_n,\mathscr{G}_n)$ denote the following graph of groups.
The graph $\Gamma_n$ has $n+1$ vertices and $n$ edges:
one vertex has valence $n$ and the $n$ edges connect this vertex,
call it $\star$, to the remaining $n$ vertices.
The graph of groups structure $\mathscr{G}_n$ assigns the trivial group to the vertex
$\star$ and to the edges and assigns the group $A$ to each vertex.
Identify the fundamental group $\pi_1(\Gamma_n,\mathscr{G}_n,\star)$ with $W_n$.

Because the automorphism $\Phi$ preserves the conjugacy classes of the $A$ in $W_n$,
the automorphism $\Phi$ can be \emph{realized} as a map $(\Gamma_n,\mathscr{G}_n,\star)\to (\Gamma_n,\mathscr{G}_n,\star)$
in the sense of \cite[Chapter 2]{MyThesis} as a subdivision followed by a morphism of graphs of groups.
Therefore by \cite[Theorem 3.9.3]{MyThesis}, which uses the algorithm of \cite{BestvinaHandel},
there is a relative train track map $f\colon (\Gamma,\mathscr{G}) \to (\Gamma,\mathscr{G})$ 
representing the outer class $\varphi$ of $\Phi$ in $\out(W_n)$.
Because $\Phi$ is polynomially-growing by assumption,
the irreducible strata of $f$ have Perron--Frobenius eigenvalue $\lambda = 1$.
By passing to an iterate of $f$,
we may assume each irreducible stratum consists of a single edge $E$,
and by subdividing and choosing orientations, we may assume that $f(E) = Eu$,
where $u$ is a path in lower strata of $(\Gamma,\mathscr{G})$.

The graph of groups $(\Gamma,\mathscr{G})$ has $n$ vertices with stabilizer $A$.
Since $\Gamma_n$ was a tree, $\Gamma$ is a tree. 
Each leaf of $\Gamma$ is one of the $n$ vertices with stabilizer $A$.
Having passed to an iterate, each such vertex is fixed by $f\colon (\Gamma,\mathscr{G}) \to (\Gamma,\mathscr{G})$,
and the action of $f$ induces an automorphism of $A$.
By inspecting the action of $f$ on these vertices, we see the iterate required is bounded by $n!$.
By passing to a further iterate, since $\aut(A)$ is finite,
we may assume this automorphism is the identity for each such vertex.
This latter step is not needed for \Cref{symmetric} nor \Cref{palindromic}.

\begin{proof}[Proof of \Cref{keylemma}]
	The idea is to use the inductive hypothesis of \cite[Definition 1.1]{Levitt}.
	Namely, consider the top stratum of $(\Gamma,\mathscr{G})$.
	It is irreducible and thus consists of a single edge $E$.
	If $E$ separates $\Gamma$, collapsing the complement of $E$ determines a free product decomposition
	$W_n = G_1 * G_2$, and basing the fundamental group at the initial vertex of $E$
	provides a lift of $f$ to an automorphism of $\pi_1(\Gamma,\mathscr{G})$,
	call it $\Phi$,
	satisfying $\Phi(G_i) = G_i$.

	If $E$ does not separate $\Gamma$, 
	then the initial vertex of $E$ is a leaf of $\Gamma$ 
	and thus one of the $n$ vertices with stabilizer $A$.
	Base the fundamental group of $(\Gamma,\mathscr{G})$ as the terminal vertex of $E$,
	call it $v$,
	and choose a path $\sigma$ in $\Gamma\setminus E$ from $v$ to $f(v)$.
	Collapsing the complement of $E$ determines a free product decomposition
	$W_n = G_1 * A$,
	and the action of $f$ on $\pi_1(\Gamma,\mathscr{G},v)$
	as $\gamma \mapsto \sigma f(\gamma)\bar \sigma$ defines an automorphism
	$\Phi$ satisfying $\Phi(G_1) = G_1$ and
	$\Phi(A) = w^{-1}Aw$ for some $w \in G_1$.
	The proof follows by induction on $n$.
\end{proof}

\bibliography{bib.bib}

\begin{thebibliography}{{Gho}18}

\bibitem[BDM09]{BehrstockDrutuMosher}
Jason Behrstock, Cornelia Dru\c{t}u, and Lee Mosher.
\newblock Thick metric spaces, relative hyperbolicity, and quasi-isometric
  rigidity.
\newblock {\em Math. Ann.}, 344(3):543--595, 2009.

\bibitem[BH92]{BestvinaHandel}
Mladen Bestvina and Michael Handel.
\newblock Train tracks and automorphisms of free groups.
\newblock {\em Ann. of Math. (2)}, 135(1):1--51, 1992.

\bibitem[BH99]{theBible}
Martin~R. Bridson and Andr\'{e} Haefliger.
\newblock {\em Metric spaces of non-positive curvature}, volume 319 of {\em
  Grundlehren der Mathematischen Wissenschaften [Fundamental Principles of
  Mathematical Sciences]}.
\newblock Springer-Verlag, Berlin, 1999.

\bibitem[BK16]{ButtonKropholler}
J.~O. Button and R.~P. Kropholler.
\newblock Nonhyperbolic free-by-cyclic and one-relator groups.
\newblock {\em New York J. Math.}, 22:755--774, 2016.

\bibitem[Bri]{BridsonProblems}
Martin~R. Bridson.
\newblock Problems concerning hyperbolic and {CAT}(0) groups.
\newblock Available at
  https://docs.google.com/file/d/0B-tup63120-GVVZqNFlTcEJmMmc/edit.

\bibitem[BS19]{BradySoroko}
Noel Brady and Ignat Soroko.
\newblock Dehn functions of subgroups of right-angled {A}rtin groups.
\newblock {\em Geom. Dedicata}, 200:197--239, 2019.

\bibitem[DL19]{DahmaniLi}
Fran{\c{c}}ois {Dahmani} and Ruoyu {Li}.
\newblock {Relative hyperbolicity for automorphisms of free products}.
\newblock Available at arXiv:1901.06760 [math.GR], January 2019.

\bibitem[Ger94]{GerstenFreeByZ}
S.~M. Gersten.
\newblock The automorphism group of a free group is not a {${\rm CAT}(0)$}
  group.
\newblock {\em Proc. Amer. Math. Soc.}, 121(4):999--1002, 1994.

\bibitem[{Gho}18]{Ghosh}
Pritam {Ghosh}.
\newblock {Relative hyperbolicity of free-by-cyclic extensions}.
\newblock Available at arXiv:1802.08670 [math.GR], February 2018.

\bibitem[GL07]{GauteroLustig}
Francois {Gautero} and Martin {Lustig}.
\newblock {The mapping-torus of a free group automorphism is hyperbolic
  relative to the canonical subgroups of polynomial growth}.
\newblock Available at arXiv:0707.0822 [math.GR], July 2007.

\bibitem[Hag19]{HagenThick}
Mark Hagen.
\newblock A remark on thickness of free-by-cyclic groups.
\newblock {\em Illinois J. Math.}, 63(4):633--643, 2019.

\bibitem[HW15]{HagenWise2}
Mark~F. Hagen and Daniel~T. Wise.
\newblock Cubulating hyperbolic free-by-cyclic groups: the general case.
\newblock {\em Geom. Funct. Anal.}, 25(1):134--179, 2015.

\bibitem[HW16]{HagenWise}
Mark~F. Hagen and Daniel~T. Wise.
\newblock Cubulating hyperbolic free-by-cyclic groups: the irreducible case.
\newblock {\em Duke Math. J.}, 165(9):1753--1813, 2016.

\bibitem[Lev09]{Levitt}
Gilbert Levitt.
\newblock Counting growth types of automorphisms of free groups.
\newblock {\em Geom. Funct. Anal.}, 19(4):1119--1146, 2009.

\bibitem[Lym20]{MyThesis}
Rylee~Alanza Lyman.
\newblock {\em Train {T}racks on {G}raphs of {G}roups and {O}uter
  {A}utomorphisms of {H}yperbolic {G}roups}.
\newblock ProQuest LLC, Ann Arbor, MI, 2020.
\newblock Thesis (Ph.D.)--Tufts University.

\bibitem[Mac02]{Macura}
Nata\v{s}a Macura.
\newblock Detour functions and quasi-isometries.
\newblock {\em Q. J. Math.}, 53(2):207--239, 2002.

\bibitem[Sam06]{Samuelson}
Peter Samuelson.
\newblock On {${\rm CAT}(0)$} structures for free-by-cyclic groups.
\newblock {\em Topology Appl.}, 153(15):2823--2833, 2006.

\end{thebibliography}
\bibliographystyle{alpha}
\end{document}